\documentclass[12pt,reqno,twoside]{amsart}
\usepackage{amsmath}
\usepackage{amssymb}
\usepackage{color}
\usepackage{hyperref}
\usepackage[english]{babel}

\usepackage{comment}
\usepackage{bm}
\usepackage{verbatim}
\usepackage{amsthm}
\usepackage{amstext}
\usepackage{amsbsy}
\usepackage{ifpdf}
\usepackage[nobysame,abbrev,alphabetic]{amsrefs}
\usepackage{enumerate}
\usepackage{amscd}
\usepackage{amsfonts}
\usepackage{graphicx}
\usepackage[all]{xy}
\usepackage{verbatim}
\usepackage{mathrsfs}

\newcommand{\N}{{\mathbb{N}}}

\newcommand{\F}{{\mathbb{F}}}

\newcommand {\ignore}[1]  {}

\newcommand{\eps}{\varepsilon}

\newcommand{\PP}{{\mathcal P}}
\newcommand{\PPdisc}{{\mathcal P^{\mathrm{(disc)}}_{L}}}

\newcommand{\df}{{\, \stackrel{\mathrm{def}}{=}\, }}

\newcommand{\Tor}{\mathbb{T}}

\newcommand{\Convn}{{\mathrm{Conv}_n}}

\newcommand{\R}{{\mathbb{R}}}
\newcommand{\TT}{{\mathbb{T}}}

\newcommand{\Z}{{\mathbb{Z}}}

\newcommand{\Dd}{{\mathcal{D}}}

\newcommand{\Vol}{{\mathrm{vol}}}

\newcommand{\Gr}{{\mathrm{Gr}}}

\newcommand{\Unif}{\mathrm{Uniform}}
\newcommand{\SM}{{\eta}}
\newcommand{\SMfp}{{\eta_{\F_p}}}

\newcommand{\SMDA}{{\Phi}}
\newcommand{\covol}{\mathrm{covol}}

\usepackage{dsfont}

\newcommand{\KK}{{\mathcal{K}}}

\newtheorem{thm}{Theorem}[section]

\newtheorem{lem}[thm]{Lemma}

\newtheorem{prop}[thm]{Proposition}

\newtheorem{cor}[thm]{Corollary}

\newtheorem{remark}[thm]{Remark}

\newtheorem{dfn}[thm]{Definition}

\newif\ifdraft\drafttrue

\title[Density of smooth lattice coverings]{Bounds on the density of smooth lattice coverings}

\author{Or Ordentlich}
\address{School of Computer Science and Engineering, Hebrew University}

\author{Oded Regev}
\address{Courant Institute of Mathematical Sciences, New York University}

\author{Barak Weiss}
\address{School of Mathematical Sciences, Tel Aviv University}

\begin{document}
	\date{\today}
	\maketitle
	
\begin{abstract}
Let $\KK$ be a convex body in $\R^n$, let $L$ be a lattice with unit covolume, and let $\eta>0$. We say that $\KK$ and $L$ form an {\em $\eta$-smooth cover} if each point $x\in\R^n$ is covered by $(1 \pm\eta)\Vol(\KK)$ translates of $\KK$ by $L$. We prove that for any positive $\sigma$ and $\eta$, asymptotically as $n\to \infty,$ for any $\KK$ of volume $n^{3+\sigma}$,  one can find a lattice $L$ for which $\KK, L$ form an $\eta$-smooth cover. Moreover, this property is satisfied with high probability for a lattice chosen randomly, according to the Haar-Siegel measure on the space of lattices. Similar results hold for random construction A lattices, albeit with a worse power law, provided that the ratio between the covering and packing radii of $\Z^n$ with respect to $\KK$ is at most polynomial in $n$.
   Our proofs rely on a recent breakthrough
   of Dhar and Dvir on the discrete Kakeya problem. 
\end{abstract}
	
\section{Introduction}
Let $\Convn$
denote the set of 
bounded convex subsets of  $\R^n$ with nonempty interior. For a lattice $L\subset \mathbb{R}^n$, convex set $\KK\in\Convn$,
and a point $x\in\mathbb{R}^n$ we denote
\begin{equation}
\label{eq: we denote}
N(L, \KK, x) \df  \left|L \cap (\KK+x)\right| = \left| \left\{y\in L \ : \ x\in y-\KK\right\}\right|.
\end{equation}
The expectation of $N(L, \KK,x)$ when $x$ is drawn uniformly from a fundamental domain for $L$ 
is $\Vol(\KK)/\covol(L)$; thus if we draw $x$ uniformly in a ball $B(0,T)$ with respect to some norm,  the expectation of $N(L, \KK,x)$ approaches $\Vol(\KK)/\covol(L)$, in the limit as $T \to \infty$.  Furthermore, we have that $\frac{N(L, \alpha\KK, x)}{\Vol(\alpha\KK)/\covol(L)}$ tends to $1$ as the dilation factor $\alpha$ grows, where the convergence is uniform in $x$. It is therefore natural to ask,  given $\KK\in\Convn$ and a lattice $L\subset\R^n$, whether the fraction $\frac{N(L, \KK, x)}{\Vol(\KK)/\covol(L)}$
 is nearly constant on $\R^n$. To that end we define the following quantity:
\begin{dfn}\label{dfn: smoothness continuous}
	The \emph{covering smoothness} of a lattice $L\subset\R^n$ with respect to a convex body $\KK\in\Convn$ is defined as
	\begin{align*}
	\SM(\KK,L)\df\sup_{x\in\R^n}\left|\frac{N(L,\KK,x)}{\Vol(\KK)/\covol(L)}-1\right|.\nonumber
	\end{align*}
\end{dfn}
Note that $\SM(\KK,L)<1$ immediately implies that $L+\KK=\R^n$. In this case the pair $(L, \KK)$ is said to be {\em a covering}. (The  reverse statement is not true --- there may be points $x$ that are covered an exceptionally large number of times by translates of $\KK$.) 
As in the abstract, if $\eta(\KK, L) < \eta$ we say that the covering given by $\KK$ and $L$ is {\em $\eta$-smooth.} 

Let $\mu_n$ denote the Haar-Siegel measure; that is, the unique probability measure on the space of lattices in $\R^n$ of unit covolume, which is invariant under volume preserving linear transformations. Our main result is that for every $\KK\in\Convn$ whose volume is polynomial in $n$, and for most lattices in $\R^n$, the covering smoothness is small.

\begin{thm}\label{thm: this result}
Let $n>25$, and let $\KK\in\Convn$. Let $\delta,\eps\in(0,1)$, and assume $\Vol(\KK)\geq c_1 \left(\frac{1}{\eps\delta}\right)^{6.5} n^3$, where $c_1=2^{66}$.
Then, for $L\sim\mu_n$ we have
\begin{align}
\Pr\left(\SM(\KK,L) \geq \eps\right) <  \delta.
\label{eq:SiegelProbBound}
\end{align}
In particular, for any positive $\eps, \sigma$, 
\begin{align}
\sup_{\KK \in \Convn, \Vol(\KK) \geq n^{3+\sigma}}\Pr \left(\eta(\KK,L) \geq \eps \right) \longrightarrow_{n\to \infty} 0.\nonumber
\end{align}
\label{thm:eta_siegel}
\end{thm}

We did not attempt to optimize the multiplicative constant $c_1$ in this result, or similar constants $c_i$ in the sequel. 

Theorem~\ref{thm:eta_siegel} and the remaining statements below might find applications in computer science. Specifically, in lattice-based cryptography, smoothing a lattice is a key idea used to hide secret information from an adversary~\cite{MicciancioR07}. Typically, one considers smoothing by a Gaussian distribution. However, for some applications 
it might be advantageous to smooth using a convex body, since sampling from a convex body like a cube can be more efficient. 
It should be noted that in many cryptographic applications, closeness in $L_1$ is sufficient, and such $L_1$-smoothness results (in fact, even $L_2$-smoothness) are often much easier to prove (see, e.g.,~\cite{Debris23}); yet there are many cases where closeness in $L_\infty$ (as in our results) leads to tighter results~\cite{BaiLRSSS18}.





It is instructive to compare Theorem \ref{thm: this result} with our previous work \cite{ORW21} on lattice coverings. Recall that for a lattice $L \subset \R^n$ of covolume one, the {\em covering density} $\Theta(\KK, L)$ is the minimal volume of a dilate  $\alpha \KK$ such that $(L, \alpha \KK)$ is a covering. One of the results of \cite{ORW21} is that
\begin{equation}\label{eq: subexponential}
\sup_{\KK \in \Convn} \inf_{L \text{ of covolume } 1} \Theta(\KK, L)
\end{equation}
grows at most quadratically in $n$ (prior to \cite{ORW21} the best known bound, due to Rogers \cite{Rogers_bound}, was superpolynomial). In fact, it was shown that for any $\delta>0$, any $\sigma>0$, any large enough $n$ and any $\KK \in \Convn,$ if $\Vol(\KK)\geq n^{2+\sigma}$ then the $\mu_n$-probability that $(L, \KK)$ is a covering is at least $1-\delta.$
Fixing 
$\eps \in (0,1)$, we deduce easily from Theorem \ref{thm:eta_siegel} the slightly weaker statement, in which  $2+\sigma$ is replaced with $3+\sigma$.  That is,  when $\Vol(\KK) > n^{2+\sigma}$, we know from \cite{ORW21} that a random $L$ gives a covering, and from Theorem \ref{thm:eta_siegel} we know that when $\Vol(\KK)> n^{3+\sigma}$ a random $L$ gives an $\varepsilon$-smooth covering. On the other hand,~\cite{CoxeterFewRogers} shows that for $\KK$ taken as the Euclidean ball and any lattice $L$ of covolume $1$, $\eta(\KK,L)\ge 1$ (and moreover, $(L,\KK)$ is not a covering) unless $\Vol(\KK)=\Omega(n)$.

While we intuitively expect covering
to become smoother as we scale up $\KK$, it turns out that in general $\alpha\mapsto\eta(\alpha\KK,L)$ is not monotonically non-increasing. To see this, take $L=\Z^n$ and $\KK=\left[0,1\right)^n$.
Then $\eta(\KK,\Z^n)=0$ yet for small $\eps>0$,  $\eta((1+\eps)\KK,\Z^n)=(2/(1+\eps))^n-1$. It is therefore natural to further define the following quantity for a lattice $L\subset\R^n$ and $\KK\in\Convn$,
	\begin{align}
	\SMDA_{\KK,L}(\eps)\df\sup\left\{\frac{\Vol(\alpha\KK)}{\covol(L)} \ : \ \alpha>0 \text{ satisfies } \SM(\alpha\KK,L)>\eps  \right\}.\nonumber
	\end{align}
In particular, for a lattice $L\subset\R^n$ of unit covolume, we have that $\SM(\alpha\KK,L)\le \eps$ for \emph{all} dilates $\alpha\KK$ of volume exceeding
$\SMDA_{\KK,L}(\eps)$. We prove the following theorem.

\begin{thm}
 Let $n> 25$, and let $\KK\in\Convn$. Let $\delta,\eps\in(0,1)$, and $c_2=2^{112}$.
Then, for $L\sim\mu_n$ we have
\begin{align}
\Pr\left(\SMDA_{\KK,L}(\eps) \geq c_2 \left(\frac{1}{\eps^2\delta}\right)^{6.5} n^{3}\right)<\delta.
\label{eq:PhiProbBound}
\end{align}
\label{thm:Phi_siegel}
\end{thm}



\subsection{Construction A lattices} \label{subsec: construction A}
In many applications in electrical engineering and computer science, integer lattices known as \emph{construction A lattices} are of interest~\cites{ConwaySloane,loeliger97}. For a prime $p$ let $\F_p$ denote the field with $p$ elements. For $r
\in \{1, \ldots, n\}$, let 
$\Gr_{n,r}(\F_p)$ denote the collection of subspaces of
dimension $r$ in $\F_p^n$, or equivalently, the rank-$r$ additive
subgroups of $\F_p^n$.  We can identify $\F_p$ with the residues 
$\{0, \ldots, p-1\}$, and thus identify $\F_p^n$ with the quotient
$\Z^n/p\Z^n$. We have a natural \emph{reduction mod $p$} homomorphism
$\pi_p: \Z^n \to
\F_p^n$, which sends each coordinate of $x \in \Z^n$ to its class
modulo $p$. Any element $S \in \Gr_{n,r}(\F_p)$ gives rise to a sub-lattice
$\pi_p^{-1}(S) \subset \Z^n$, which contains $p\Z^n$ as a
subgroup of index $p^{r}$,  and with $\pi_p^{-1}(S) / p\Z^n$ isomorphic
as an abelian group to $S \cong \prod_1^r \Z/p\Z$. 
The ensemble of lattices obtained by drawing $S\sim\Unif(\Gr_{n,r}(\F_p))$ and setting $L=\frac{1}{p}\cdot\pi_p^{-1}(S)$ is called the \emph{random $(p,r)$ construction A ensemble}\footnote{Some authors define the $(p,r)$ random construction $A$ ensemble slightly differently, taking  $S'=\mathrm{span}_{\F_p}(v_1,\ldots,v_r)$ and $v_1,\ldots,v_r\stackrel{i.i.d.}{\sim}\Unif(\F_p^n)$ and $L=\frac{1}{p}\cdot\pi_p^{-1}(S')$. Since $\Pr(S'\notin \Gr_{n,r}(\F_p))\leq p^{r-n}$ and, moreover, $S'$ is conditionally uniform on $\Gr_{n,r}(\F_p)$ under the event $S'\in \Gr_{n,r}(\F_p)$, we have that the total variation distance between the distributions corresponding to the two definitions is at most $p^{r-n}$.}
and such lattices are called \emph{$(p,r)$ construction A lattices}. 

Theorem~\ref{thm:eta_siegel} holds for any $\KK\in\Convn$, with uniform constants. However, if $\KK$ is such that $N(\KK,\Z^n,0)$ is large (for example $\KK = \prod_1^{n-1} [\eps^{-1}, \eps^{-1}] \times [-\eps^{n-1}, \eps^{n-1}]$ for $\eps$ small), it will not be smoothed by applying construction A, unless $p$ and $r$ are  large (depending on $\KK$). Thus  our results for construction A lattices depend on $\KK$. The dependence arises via the ratio between the {\em covering radius} and {\em packing radius} of $\KK$ with respect to $\Z^n$. Namely, for a convex body $\KK\in\Convn$ and a lattice $L\subset\R^n$ we denote by $\mathrm{r}_{\text{cov},\KK}(L)$ the  infimum of $\alpha$ for which $(L,\alpha\KK)$ is a covering, and by $\mathrm{r}_{\text{pack},\KK}(L)$ the supremum of $\alpha$ for which $(L, \alpha\KK)$ is a {\em packing}, i.e., the translates $\{\ell+\alpha\KK  : \ell\in L \}$ are disjoint.  We denote by 
\begin{align}
\rho_{\KK}(L) \df \frac{\mathrm{r}_{\text{cov},\KK}(L)}{\mathrm{r}_{\text{pack},\KK}(L)} \label{eq:covpackratio}
\end{align}
the ratio between the covering and the packing radius.
 We show that for $\KK\in\Convn$ for which both $\Vol(\KK)$ and $\rho_\KK(\Z^n)$ are polynomial in $n$, for a typical  construction A lattice with adequately tuned $p$, $r$, scaled to have unit
 covolume, the covering smoothness is small.


\begin{thm}
Let $n> 25$, and let $\KK\in\Convn$ and $b>0$ satisfy
\begin{align}
    \label{eq: new bounds n b}
0\leq b\leq \frac{n}{2 \log_2 n} \ \ \text{ and } \ \ \rho_\KK(\Z^n) < n^b.
\end{align}
Let $\delta,\eps\in(0,1)$, and assume  $\Vol(\KK)\geq c_{3} \left(\frac{1}{\eps\delta} \right)^{6} n^{3(1+2b)}$, where $c_{3}=e\cdot 2^{33}$. 
Let $p$ be a prime number satisfying 
\begin{align}
\frac{1024}{(\eps\delta)^2}n^{1+2b}\leq p\leq \frac{2048}{(\eps\delta)^2}n^{1+2b},
\label{eq:pdefconsta}
\end{align}
and $r=3+\left\lceil \frac{n}{\log p}\left(b\log n+ \log 3 \right)\right\rceil$. Then, if $L$ is drawn from the $(p,r)$ random construction A ensemble (so that $\covol(p^{r/n}L)=1$), we have
\begin{align*}
\Pr\left(\SM(\KK,p^{r/n}L) \geq \eps\right) <  \delta \; .
\end{align*}
\label{thm:eta_constA}
\end{thm}


An important special case is the Euclidean ball, namely $\KK=\mathcal{B}_n=\{x\in\R^n \ : \ \|x\|_2\leq 1 \}$. It is well known (and easy to see) that \begin{align*}
\mathrm{r}_{\text{pack},\mathcal{B}_n}(\Z^n)&=\frac12 \ \ \ \ \text{ and } \ \  \mathrm{r}_{\text{cov},\mathcal{B}_n}(\Z^n)=\frac{\sqrt{n}}{2},
\end{align*}
which gives
\begin{align*}
\rho_{\mathcal{B}_n}(\Z^n)=n^{1/2}.
\end{align*}
Thus, the following is an immediate consequence of Theorem~\ref{thm:eta_constA}.

\begin{cor}
Let $n>25$, and let $\delta,\eps\in(0,1)$, and for $\alpha>0$, assume 
$$\Vol(\alpha \mathcal{B}_n)\geq c_{3} \left(\frac{1}{\eps\delta} \right)^{6} n^{6},$$ 
where $c_{3}=e\cdot 2^{33}$. 
Let $p$ be a prime number satisfying 
\begin{align*}
\frac{1024}{(\eps\delta)^2}n^{2}\leq p\leq \frac{2048}{(\eps\delta)^2}n^{2},
\end{align*}
and $r=3+\left\lceil \frac{n}{\log p}\left(\frac{1}{2}\log 9n\right)\right\rceil$. Then if $L$ is drawn from the $(p,r)$ random construction A ensemble (so that $\covol(p^{r/n}L)=1$), we have
\begin{align*}
\Pr\left(\SM(\alpha \mathcal{B}_n,p^{r/n}L) \geq \eps\right) <  \delta \; .
\end{align*}
\label{cor:constAball}
\end{cor} 

Similarly to the case where $L$ is drawn at random according to the distribution $\mu_n$, we can also show that for $L$ drawn from the $(p,r)$ random construction A ensemble, for $\KK\in\Convn$ with $\rho_{\KK}(\Z^n)$ polynomial in $n$, we have that with high probability $\SMDA_{\KK,L}(\eps)$ is also polynomial in $n$, provided that $p$ and $r$ are chosen adequately.

\begin{thm}
Let $n>25$, and let $\KK\in\Convn$ and $b>0$ satisfy~\eqref{eq: new bounds n b}.
Let $\delta,\eps\in(0,1)$, and  $c_{4}=e\cdot 2^{57}$. 
Let $p$ be a prime number satisfying 
\begin{align}
\frac{2^{18}}{(\eps^2\delta)^2}n^{3+2b}\leq p\leq \frac{2^{19}}{(\eps^2\delta)^2}n^{3+2b},
\label{eq:pdefconstaPhi}
\end{align}
and $r=3+\left\lceil \frac{n}{\log p}\left(b\log n+ \log 3 \right)\right\rceil$.
Then, if $L$ is drawn from the $(p,r)$ random construction A ensemble, we have
\begin{align}
\Pr\left(\SMDA_{\KK,L}(\eps)>c_{4} \left(\frac{1}{\eps^2\delta} \right)^{6} n^{9+6b}\right)\leq \delta.
\end{align}
\label{thm:Phi_constA}
\end{thm}

\subsection{Non-lattice smooth coverings}
If one relaxes the requirement that $L$ is a lattice, it is slightly more complicated to define smooth covers, but much easier to construct them. 

Let $L \subset \R^n$ be a discrete set (not necessarily a lattice). We continue to use the notation $N(L, \KK, x)$ defined in \eqref{eq: we denote}. Let $B(0,T)$ denote the ball of radius $T$ around the origin with respect to the Euclidean norm, and define the {\em asymptotic upper density} of $L$ by
\begin{equation}\label{eq: def D}
D(L) \df \limsup_{T \to \infty} \frac{|B(0,T) \cap L|}{\Vol(B(0,T))}.
\end{equation}
If the limit in \eqref{eq: def D} exists we will say that {\em $L$ has an asymptotic density}. 
Note that lattices have an asymptotic density given by $D(L) =\covol(L)^{-1}$.
Now for $\KK \in \Convn$, and $L$ as above, we set 
	\begin{align*}
	\SM(\KK,L)\df\sup_{x\in\R^n}\left|\frac{N(L,\KK,x)}{\Vol(\KK) \, D(L)}-1\right|.\nonumber
	\end{align*}
With this notation we have:
\begin{thm}\label{thm: non lattice}
For any $0<\eps<1$, $n\geq 20$, and any $\KK \in \Convn$ 
there is a discrete set $L \subset \R^n$ which has an asymptotic density, satisfying 
\begin{equation}\label{eq: desired bounds}
\Vol(\KK) D(L) \leq \frac{14}{\eps^2}\left(n\log n+n\log\frac{1280}{\eps}+2\right) \ \ \text{ and } \ \ \eta(\KK, L)<\eps.
\end{equation}
\end{thm}
We remark that the set $L$ constructed in the proof of Theorem \ref{thm: non lattice} is \emph{periodic}, i.e., consists of finitely many translates of a lattice in $\R^n$. 
We also remark that Theorem~\ref{thm: non lattice} can be derived by modifying the proof of Erd{\H{o}}s and Rogers~\cite{ErdosRogers}, who proved a closely related statement. For completeness, we include a proof in Section~\ref{sec:nonlattice} that follows the proof structure of  our main theorem. 

\subsection{Acknowledgements}
The authors are grateful to Bo'az Klartag for suggesting the question of seeking smooth lattice coverings, to Manik Dhar and Ze'ev Dvir for sharing an early draft of their result, and to Chris Peikert for useful comments. 
The first author is supported by ISF 1641/21, the second author is supported by a Simons Investigator Award from the Simons Foundation, and the third author is supported by ISF 2019/19 and ISF-NSFC 3739/21.
\section{Techniques and Notation}
The main results of this paper follow from the somewhat technical Theorem \ref{thm:packtocov}. This result is derived in turn from a new result of Dhar and Dvir (Theorem \ref{thm: DD}), which is a crucial input to this paper. 
In this section we introduce notations, give a brief overview of our approach, and state Theorem~\ref{thm: DD}.

For a
lattice $L \subset\R^n$ let 
$\TT_L \df \R^n/L$ be the quotient torus, let $m_L$ be the Haar probability
measure on $\TT_L$, and let $\pi_L : \R^n \to
\TT_L$ be the quotient map.  
Let $v_1, \ldots, v_n$ be the generators of $L$ given by the columns of $g$
so that the parallellepiped 
$$
	\PP_L = \left\{\sum a_i v_i : \forall i, \ 0 \leq a_i < 1 \right \}
$$
is a fundamental
domain for $\R^n/L$. Define the discrete `net' 
\begin{equation}
    \label{eq: def P disc}
\PPdisc \df \left\{
	\sum a_i v_i \in \PP_L:  a_i \in  \left\{ 0, \frac{1}{p}, \ldots, 1-\frac{1}{p}
	\right \} \right\}
\end{equation}
and set
$$\frac{L}{p} \df 
\frac{1}{p} \cdot L.
$$
Then the elements of $\PPdisc  $ are coset representatives for the inclusion $L \subset \frac{L}{p},$ and there is an isomorphism (as abelian groups) $\PPdisc \cong\F_p^n.$

Next, we introduce a well-studied technique for randomly choosing lattices. 
Given $L=g\Z^n$, where $g$ is an invertible $n \times n$ matrix, and given  $S\in\Gr_{n,r}(\F_p)$, we define the  super-lattice $L(S)\supset L$ as
\begin{align}
L(S)\df \frac{1}{p} \cdot g \pi_p^{-1}(S).
\label{eq:LSdef}
\end{align}
Notice that the scaled-up version $p^{r/n} \cdot L(S)$ of $L(S)$ is of the same covolume as $L$.
The assignment 
\begin{equation} \label{eq: Hecke friends} L \ \mapsto \ \left\{p^{r/n} L(S):S \in \Gr_{n,r}(\F_p) \right\}
\end{equation}
is a special case of the so-called {\em Hecke correspondence.}
Note that the individual lattice $L(S)$ also depends on the initial choice of $g$ for which $L = g\Z^n,$ but the collection on the right-hand side of \eqref{eq: Hecke friends} does not. 
Also note that Construction A lattices are a special case of this construction (up to scaling), starting with $L = \Z^n$.

Given a convex body $\KK$, our goal is to find a lattice for which the covering smoothness is small. 
We will choose this lattice to be $L(S)$ for a randomly chosen $S \in \Gr_{n,r}(\F_p)$ for some $r$ and $p$, and where $L$ is a lattice for which we have a reasonable bound on $\rho_{\KK}(L)$. 
For instance, in the proof of Theorem~\ref{thm: this result}, we will take $L$ to be a randomly chosen lattice according to the Haar-Siegel measure $\mu_n$, which has a small $\rho_{\KK}$ (Proposition~\ref{prop: 3.6}); 
importantly, by~\cite{ORW21}*{Proposition 2.1}, $L(S)$ is also distributed according to $\mu_n$ (up to scaling), as needed for the conclusion of Theorem~\ref{thm: this result}. 

By rescaling $\KK$ we may assume that $L$ forms a packing with respect to $\KK$ and a covering with respect to the dilate $\rho_{\KK}(L) \KK$. Recall that our goal is to show that the function 
\begin{equation}\label{eq: function to be made constant}
\Tor_{L(S)} \to \N, \ \ \  \ x \mapsto |(x+L(S))\cap \KK|
\end{equation}
is uniformly close to a constant function. By a discretization procedure (see Proposition \ref{prop: discretization}), it will be sufficient to show that the restriction of the function in~\eqref{eq: function to be made constant} to $\mathcal{P}^{(\mathrm{disc})}_L$ is close to a constant function. This uses the fact that $L$ is a covering with respect to $\rho_{\KK}(L) \KK$ and assumes that $p$ is chosen sufficiently large with respect to $\rho_{\KK}(L)$.

The final step in the proof is to reduce the problem to an analogous problem in $\F_p^n$. 
Denoting by $A\subset \F_p^n$ the set $\pi_{L}\big(\frac{L}{p}\cap\KK\big)$ viewed as a subset of $\F_p^n$, we have that for any $x \in \mathcal{P}^{(\mathrm{disc})}_L$, 
\begin{align}
   |(x+L(S)) \cap \KK| = 
   \Big|\pi_L\Big((x+L(S)) \cap \KK\Big)\Big| = 
   |(x+S)\cap A|\; ,  \label{eq:cont to discrete}
\end{align}
where the first equality uses the assumption that $L$ forms a packing with respect to $\KK$ (and so $\pi_L$ is injective on $\KK$) and on the right-hand side we think of $x$ as being in $\F_p^n$. 
Thus our problem reduces to showing that a randomly chosen $S$ leads to a smooth covering of $\F_p^n$  by the $S$-translates of $A$. This is precisely the problem addressed by Dhar and Dvir.

To state their result we need the following discrete analogue of the covering smoothness:
 
\begin{dfn}\label{dfn: smoothness discrete}
	Let $p$ be a prime number. The {\em smoothness of a set $S\subset\F_p^n$ with respect to a set $A\subset\F_p^n$} is defined as 
	\begin{align*}
	\SMfp(A,S)\df\sup_{x\in\F_p^n}\left|\frac{|(x+S)\cap A|}{|S|\cdot|A|\cdot p^{-n}}-1\right|.\nonumber
	\end{align*}
\end{dfn}
With this notation we have:\footnote{The precise statement given in~\cite{dd22}*{Theorem III.3} deals only with the smallest possible choice of $r=3+n-\lfloor\log_p |A|\rfloor=\lceil 3+n-\log_p |A|\rceil$. However, any larger choice of $r$ also works, since, as noted in~\cite{dd22}, if a subspace $S$ is $\tau$-shift-balanced, and $S'$ is a subspace containing $S$, then $S'$ is also $\tau$-shift-balanced.}
\begin{thm}[Dhar and Dvir, Theorem III.3 in~\cite{dd22}]\label{thm: DD}
	Let $n\geq 5$, let $\delta,\tau\in(0,1)$ and let $p$ be a prime number satisfying $p>64n/(\tau\delta)^2$. Let $A\subset \F_p^n$ and let $4\leq r\leq n$ be an integer satisfying $r>3+n-\log_p |A|$. Then for $S\sim \Unif\left(\Gr_{n,r}(\F_p)\right)$ we have
	\begin{align}
	\Pr(\SMfp(A,S)>\tau)<\delta.\label{eq:ddbound}
	\end{align}
	\label{thm:dd22}
\end{thm}
To summarize, by~\eqref{eq:cont to discrete} and the discretization argument, if the conclusion of Theorem~\ref{thm: DD} holds with $A=\pi_{L}\big(\frac{L}{p}\cap\KK\big)$, then we obtain the desired result, namely, that with high probability, $L(S)$ is a smooth covering for $\KK$ of density
\[
\frac{\Vol(\KK)}{\covol(L(S))}=
p^r\frac{\Vol(\KK)}{\covol(L)} \; .
\]
Noting that $|A|\approx p^n \frac{\Vol(\KK)}{\covol(L)}$ (Lemma~\ref{lem:Gauss}) and taking  $r=3+n-\log_p|A|$ (ignoring here the technicality of $r$ having to be an integer) we get a density of about $p^3$. Finally, we need to choose $p$ to satisfy the conditions of Theorem~\ref{thm: DD} and be large enough compared to  $\rho_{\KK}(L)$ for the discretization argument to work. 


We end this section by noting that improvements to Theorem~\ref{thm: DD} will yield improvements in our results. We make this precise in Remark~\ref{rem:ddparams}.

\section{Proofs of Main Results for Lattices}
\subsection{Discretization}
For subsets $A, B \subset \R^n$ and $c \in \R$ we denote as usual $$A+B \df \{a+b: a \in A, b \in B\}, \ \ cA \df \{ca: a \in A\}.$$
\begin{prop}\label{prop: discretization}
	Let $L'\subset\mathbb{R}^n$ be a discrete subset of $\R^n$ such that $L'=-L',$ let
        $\KK\in\Convn$, and assume $0<\rho<1$ is
        such that $L'+\rho \KK=\mathbb{R}^n$. Then for any $x \in \R^n$ there are $y_1, y_2 
        \in L'$ such that 
  \begin{equation}\label{eq: containment both sides}
  (1-\rho)\KK +y_1 \subset \KK+x \subset (1+\rho)\KK+y_2.
  \end{equation}
  %
\end{prop}

\begin{proof}
The convexity of $\KK$ implies that for any positive $\alpha$ and $\beta$ we have 
\begin{align}
\alpha \KK + \beta \KK = 
(\alpha + \beta) \left\{\frac{\alpha}{\alpha + \beta}k_1 + \frac{\beta}{\alpha+\beta} k_2: k_1, k_2 \in \KK \right\}  =(\alpha + \beta)\KK.  \label{eq:convsetsum}  
\end{align}
For the containment on the right-hand side of \eqref{eq: containment both sides},  since $L'+\rho \KK=\mathbb{R}^n$, for any 
                  $x\in\mathbb{R}^n$ there is $y_2\in L'$ such
                  that $x\in y_2+\rho \KK$. Thus if $y\in \KK +x $ 
                  then $y\in \KK + \rho \KK + y_2 =(1+\rho)\KK +y_2.$
 

 For the other containment, since $\R^n = -\R^n = -(L'+\rho \KK) = L' - \rho \KK$, there is $y_1 \in L'$ such that $x \in y_1 - \rho \KK,$ and thus $y_1 \in x+\rho \KK$. Now if 
 $y\in (1-\rho) \KK + y_1$ 
then  $y \in (1-\rho) \KK +\rho \KK +x = \KK+x. $
\end{proof}

As an immediate corollary we see that if $L', \, \KK$ and $\rho$ satisfy the conditions of Proposition \ref{prop: discretization} and $L \subset \R^k$ is a discrete subset, then: 
	\begin{enumerate}
		\item \label{item: 1}
  $\forall x\in\mathbb{R}^n  \ \exists x'\in L'$
                  such that $N(L,\KK,x)\leq N(L,(1+\rho)\KK ,x') $; 
		\item \label{item:disc2cont_min}$\forall x\in\mathbb{R}^n  \ \exists x'\in L'$
                  such that $N(L,\KK,x)\geq N(L,(1-\rho)\KK ,x') $. 
	\end{enumerate}
Consequently, we have:

\begin{lem}\label{lem: discretization}
	Let $L', \KK$ and $\rho$ satisfy the conditions of Proposition \ref{prop: discretization} and let $L \subset\mathbb{R}^n$ be a discrete subset. Then: 
	\begin{enumerate}
		\item $\max_{x\in\mathbb{R}^n}N(L,\KK,x)\leq \max_{x'\in L'}N(L,(1+\rho)\KK ,x') $;
		\item $\min_{x\in\mathbb{R}^n} N(L,\KK,x)\geq
		\min_{x'\in L'} N(L,(1-\rho)\KK ,x') $. 
	\end{enumerate}
\label{lem:dic2cont}
\end{lem}

The following standard lemma will be useful. We give the proof for lack of a suitable reference. 

\begin{lem}\label{lem: analogy with}
Let $L\subset\R^n$ be a lattice and $\Dd\in\Convn$, and assume that $L+\beta\Dd=\R^n$ for some $\beta\in (0,1)$. Then
\begin{align}
    (1-\beta)^n\frac{\Vol(\Dd)}{\covol(L)}\leq |L\cap\Dd|\leq (1+\beta)^n\frac{\Vol(\Dd)}{\covol(L)}.
    \label{eq:volpluscountest}
\end{align}
\label{lem:Gauss}
\end{lem}

\begin{proof}
Let 
$N_{\Dd}\df L\cap \Dd,
$
and let $\mathcal{V}$ be a fundamental domain for $L$ contained in $\beta \Dd$; that is, a measurable set such that for each $x \in \R^n$ there is exactly one $\ell \in L$ for which $x \in \ell + \mathcal{V}$. Such a fundamental domain exists since $L+\beta \Dd=\R^n$. Define the sets 
\begin{align*}
\mathcal{S}_+& \df N_{\Dd}+\mathcal{V}\\
\mathcal{S}_-& \df N_{\Dd}+(-\mathcal{V}).
\end{align*}
We have that $\Vol(\mathcal{S}_+)=\Vol(\mathcal{S}_-)=|N_{\Dd}|\covol(L)$. Thus, to establish~\eqref{eq:volpluscountest}, it suffices to show that
\begin{align*}
\mathcal{S}_+&\subseteq (1+\beta)\Dd\\
\mathcal{S}_-&\supseteq (1-\beta)\Dd.
\end{align*} 
The inclusion $\mathcal{S}_+\subseteq (1+\beta)\Dd$ follows from $N_{\Dd}\subset \Dd$, $\mathcal{V}\subseteq \beta \Dd$ and the convexity of $\Dd$, using~\eqref{eq:convsetsum}. To see that $(1-\beta)\Dd\subseteq \mathcal{S}_-$, 
since $-\mathcal{V}$ is also a fundamental domain for $L$, 
for any $x\in (1-\beta)\Dd$ there is $y\in L$ such that $x\in y-\mathcal{V}$. Thus, $y\in x+\mathcal{V}\subset (1-\beta) \Dd + \beta \Dd = \Dd$, and consequently $y\in N_{\Dd}$.
\end{proof}

\subsection{From packing to smooth covering}
Our analysis of $\eta (\KK,L(S))$ (the covering smoothness of a lattice, Definition \ref{dfn: smoothness continuous}) relies on the analysis of $\eta_{\F_p}(A,S)$ (the analogous discrete  smoothness, Definition \ref{dfn: smoothness discrete}), where $A$ is a discrete analogue of the projection of $\KK$ modulo $L$, and $S$ is a randomly chosen subspace of $\F_p^n \cong \PPdisc$ of dimension $r$, where $r$ will be carefully  chosen.
 
We now derive our main technical statement from Theorem~\ref{thm:dd22}.

\begin{thm}\label{thm:packtocov}
Let $n\geq 5$, let $\delta,\tau\in(0,1)$ and let $p$ be a prime number satisfying 
$ p > 64n/(\tau\delta)^2$. 
Let $L\subset\R^n$ be a lattice and let $\KK\in\Convn$, and assume that there is some real number $1<c<\frac{p}{2n}$ such that $\big(L, \big(1+\frac{c}{p}\big)\KK \big)$ is a packing and $(L,c\KK)$ is a covering. 
Let $4\leq r\leq n$ be an integer satisfying 
\begin{align}\label{eq:assumption on r}
r>3+ \log_p\frac{\covol(L)}{\Vol(\KK)}+3c \, \frac{n}{p\log p}\, ,   
\end{align}
and let
$S\sim \Unif\left(\Gr_{n,r}(\F_p)\right)$. Then
\begin{align}\label{eq: conclusion probability bound}
\Pr\left(\SM(\KK,L(S)) \geq \tau+8c\, \frac{n}{p} \right)< 2\delta.
\end{align}
\end{thm}

\begin{proof}
Set $\rho \df \frac{c}{p}$, so that $\rho\in\left(0,\frac{1}{2n}\right)$ and $(\frac{L}{p}, \rho \KK)$ is a covering. Set $\Dd \df (1+\rho)\KK$ and $\beta \df \frac{\rho}{1+\rho}$, so that $\beta \Dd = \rho \KK$ and we can apply the right-hand side of \eqref{eq:volpluscountest} to obtain
\begin{align}
\left|\frac{L}{p}\cap (1+\rho)\KK\right|\leq
(1+\beta)^n
\cdot \frac{\Vol(\Dd)}{\covol(\frac{L}{p})}
=
(1+2\rho)^n\frac{\Vol(\KK)}{\covol(L)}p^n.
\label{eq:volpluscountest1}
\end{align}
Similarly, by setting $\Dd \df (1-\rho)\KK$ and $\beta \df \frac{\rho}{1-\rho}$, we have
\begin{align}
(1-2\rho)^n\frac{\Vol(\KK)}{\covol(L)} p^n\leq \left|\frac{L}{p}\cap (1-\rho)\KK\right|.
\label{eq:volminuscountest}
\end{align}

Write $L = g\Z^n$ and for any $S\in\Gr_{n,r}(\F_p)$,  define $L(S)$ by \eqref{eq:LSdef} and define $\PPdisc$ by \eqref{eq: def P disc}, where the $v_i$ are the columns of $g$. 
Denote by $A_0$ (respectively, $A_1$) the set of all points in $\PPdisc$ covered by $L+(1-\rho)\KK$ (respectively, $L+(1+\rho)\KK$), viewed as elements of $\F_p^n$. Since $L$ forms a packing with respect to $(1+\rho)\KK$ (and also with respect to $(1-\rho)\KK$), the restriction of the projection $\pi_L: \R^n \to \Tor_L$ to  $(1+\rho)\KK$ (and thus to $(1-\rho)\KK$) is injective. Thus, by~\eqref{eq:volpluscountest1} and ~\eqref{eq:volminuscountest} we have 
\begin{align}
(1-2\rho)^n\frac{\Vol(\KK)}{\covol(L)}p^n\leq |A_i|\leq (1+2\rho)^n\frac{\Vol(\KK)}{\covol(L)}p^n, \ \ \ i=0,1,
\label{eq:sizeAbounds}
\end{align}
and each point in $A_0$ (or $A_1$) is covered exactly once.


For any $x\in\PPdisc$ we have that
\begin{align*}
N(L(S),(1-\rho)\KK,x)&=|(x+S)\cap A_0|,\\
N(L(S),(1+\rho)\KK,x)&=|(x+S)\cap A_1|,
\end{align*}
where, with some abuse of notation, on the left-hand side we treat $x$ as a vector in $\R^n$ and on the right-hand side as an element of $\F_p^n \cong 
\PPdisc$. Let 
\begin{align} 
E_0&\df\left\{S\in \Gr_{n,r}(\F_p) \ : \   \SMfp(A_0,S)>\tau  \right\}\label{eq:E0def}\\ 
E_1&\df\left\{S\in \Gr_{n,r}(\F_p) \ : \   \SMfp(A_1,S)>\tau  \right\},\label{eq:E1def}
\end{align}
and $E=E_0\cup E_1$. For all $S\in E^{c}$ and  $x\in\PPdisc$ we have
\begin{align*}
N\left(L(S),(1-\rho)\KK,x \right)&\geq (1-\tau)|S|\cdot |A_0|\cdot p^{-n}\\
&\geq (1-\tau) \frac{|S|\cdot(1-2\rho)^n\Vol(\KK)}{\covol(L)}\\
&=(1-\tau)(1-2\rho)^n \frac{\Vol(\KK)}{\covol(L(S))},
\end{align*}
and 
\begin{align*}
N(L(S),(1+\rho)\KK,x)&\leq (1+\tau)|S|\cdot |A_1|\cdot p^{-n}\\
&\leq (1+\tau) \frac{|S|\cdot(1+2\rho)^n\Vol(\KK)}{\covol(L)}\\
&=(1+\tau)(1+2\rho)^n \frac{\Vol(\KK)}{\covol(L(S))}.
\end{align*}
Let $L'=\frac{1}{p}L=L+\PPdisc$. By assumption,  $(L',\rho \KK)$ is a covering. Thus, by part~\eqref{item:disc2cont_min} of Lemma~\ref{lem:dic2cont} we have that for all $S\in E^c$,
\begin{align}
\min_{x\in\mathbb{R}^n} N(L(S),\KK,x)
&\geq
\min_{x'\in L'} N \left(L(S),\left(1-\rho\right)\KK ,x' \right) \nonumber \\
&=\min_{x'\in \PPdisc}N(L(S),\left(1-\rho\right)\KK ,x') \nonumber \\
&\geq (1-\tau)(1-2\rho)^n \frac{\Vol(\KK)}{\covol(L(S))},\label{eq:densityLB}
\end{align}
and
\begin{align}
\max_{x\in\mathbb{R}^n}N(L(S),\KK,x)
&\leq \max_{x'\in L'}N\left(L(S),\left(1+\rho\right)\KK ,x' \right) \nonumber \\
&=\max_{x'\in \PPdisc}N\left(L(S),\left(1+\rho\right)\KK ,x'\right) \nonumber \\
&\leq (1+\tau)(1+2\rho)^n \frac{\Vol(\KK)}{\covol(L(S))}.\label{eq:densityUB}
\end{align}
Combining~\eqref{eq:densityLB} and~\eqref{eq:densityUB}, we see that for any $S\in E^{c}$ we have
\begin{align}
\SM(\KK,L(S))&=\sup_{x\in\R^n}\left|\frac{N(L(S),\KK,x)}{\Vol(\KK)/\covol(L(S))}-1\right| \nonumber \\
&\leq\max\left\{1-(1-\tau)(1-2\rho)^n,(1+\tau)(1+2\rho)^n-1  \right\}\nonumber\\
&\leq\max\left\{\tau+(1-\tau)2\rho n,\tau+(1+\tau)4\rho n \right\} \label{eq:boundingexponents} \\
&<\tau+8\rho n, \nonumber
\end{align}
where in~\eqref{eq:boundingexponents} we have used the basic bounds (for $0<2\rho<1/n$)
\begin{align}
1-(1-2\rho)^n& <  2n\rho,\label{eq:1mrhoLB}\\
(1+2\rho)^n-1& < 4n\rho. \nonumber 
\end{align}
Thus, we have shown that \begin{align}
\SM(\KK,L(S)) < \tau+8\rho n,~~ \forall S\in E^c.
\label{eq:etaEc}
\end{align}
Using our assumption on $r$ in~\eqref{eq:assumption on r} and the lower bound in~\eqref{eq:sizeAbounds}, we have that for $i=0,1$, 
\begin{align*}
r&>
3+\log_p\frac{\covol(L)}{\Vol(\KK)}+\frac{3\rho n}{\log p} \\
&\geq 
3+\log_p\left(\frac{p^n}{|A_i|}(1-2\rho)^n\right)+\frac{3\rho n}{\log p} \\
&= 3+ n-\log_p|A_i|+n\log_p(1-2\rho)+\frac{3\rho n}{\log p} \\
&\ge 3+n-\log_p|A_i|  \; .
\end{align*}
Here, we have used the inequality
$$-\log_p(1-2\rho)=\log_p\left(1+\frac{2\rho}{1-2\rho}\right)\leq \frac{1}{(1-2\rho)\log p}\cdot 2\rho<\frac{3\rho}{\log p},$$
where the last inequality follows from $\rho<1/2n$ and $n\geq 5$. 
Thus, we may invoke Theorem~\ref{thm:dd22} to obtain
\begin{align*}
\Pr(S\in E_i) < \delta,  \ \ i=0,1,
\end{align*}
and therefore, by the union bound, 
\begin{align*}
\Pr(S\in E) < 2\delta,
\end{align*}
establishing our claim.
\end{proof}

In the sequel we will use the following  convenient consequence of Theorem \ref{thm:packtocov}. 
\begin{cor}
Let $n> 25$, $L\subset\R^n$ be a lattice with $\covol(L)=1$, and $\delta,\tau\in(0,1)$.
Also let $\KK\in\Convn$ satisfy $\rho_{\KK}(L)\leq \bar{\rho}$ for some $1 \le \bar{\rho}<\left(\frac{2}{\delta}\right)^{\frac{n}{2}}$ and 
$\Vol(\KK)>c_5 \bar{\rho}^6\left(\frac{1}{\tau\delta}\right)^6 n^3$, where $c_5=e\cdot \left(128  \right)^3$.
Then, for any prime number $p$ satisfying 
    \begin{align}
        \max\left\{\frac{64\bar{\rho}^2 n}{(\tau\delta)^2},\left((2.5)^{-n}\cdot\Vol(\KK)\right)^{1/3}\right\}<p< (e^{-1}\cdot\Vol(\KK))^{1/3},
        \label{eq:primelimits}
    \end{align}
denoting $r=3+\left\lceil\frac{n\log 3\bar{\rho}}{\log p} \right\rceil$, we have that for $S\sim \Unif\left(\Gr_{n,r}(\F_p)\right)$ (so that $\covol(p^{r/n} L(S))=1$),
\begin{align}
\Pr\left(\SM(\KK,p^{r/n} L(S)) \geq 2\tau\right) <  2\delta.
\label{eq:taudeltaProbAlternative}
\end{align}
\label{thm:inflatedeflateAlternative}
\end{cor}

We remark that the statement is not vacuous, i.e., there exists a prime number $p$ satisfying~\eqref{eq:primelimits}. To see this, note that  
\begin{align*}
&\frac{ (e^{-1}\cdot\Vol(\KK))^{1/3}}{\max\left\{\frac{64\bar{\rho}^2 n}{(\tau\delta)^2},((2.5)^{-n}\cdot \Vol(\KK))^{1/3}\right\}}\\
=&\min\left\{\frac{ (e^{-1}\cdot\Vol(\KK))^{1/3}}{\frac{64\bar{\rho}^2 n}{(\tau\delta)^2}},\frac{ (e^{-1}\cdot\Vol(\KK))^{1/3}}{((2.5)^{-n}\cdot \Vol(\KK))^{1/3}} \right\}\geq 2,    
\end{align*}
and therefore, by Bertrand's postulate, there must exist a prime number satisfying~\eqref{eq:primelimits}. Recalling that $p>\frac{64 \bar{\rho}^2n}{(\tau\delta)^2}>(3\bar{\rho})^2$, it also holds that $r<n$ since
\begin{align*}
\frac{n\log 3\bar{\rho}}{\log p}\leq  \frac{n\log 3\bar{\rho}}{2\log 3\bar{\rho}}=\frac{n}{2}.   
\end{align*}

\begin{proof}[Proof of Corollary~\ref{thm:inflatedeflateAlternative}]
We show that with the parameters above, the conditions of Theorem~\ref{thm:packtocov} hold for the lattice $p^{r/n}L$ with $c=3 p^{\frac{1}{n}} \bar{\rho}^2 $. To that end, first note that $p>\frac{64 \bar{\rho}^2n}{(\tau\delta)^2}\geq \frac{64 n}{(\tau\delta)^2}$ by definition.  Furthermore, we have that
\begin{align}
\frac{8cn}{p}&= 
24\bar{\rho}^2 p^{-(1-\frac{1}{n})}n\nonumber\\
&\leq 24\bar{\rho}^2  \left(\frac{64\bar{\rho}^2 n}{(\tau\delta)^2}\right)^{-(1-\frac{1}{n})}n\nonumber\\
&=\frac{24\bar{\rho}^{\frac{2}{n}}n^{1/n}}{64^{1-1/n}} (\tau\delta)^{2(1-\frac{1}{n})}\nonumber\\
&< \frac{\bar{\rho}^{\frac{2}{n}}}{2} (\tau\delta)^{2(1-\frac{1}{n})}\nonumber\\
&< \frac{\left(\bar{\rho}\delta^{\frac{n}{2}}\right)^{\frac{2}{n}}}{2}\tau\nonumber\\
&<\tau,\label{eq:8cnbbound}
\end{align}
and in particular, this implies that $c<\frac{p}{2n}$.  
Next we lower bound the packing radius as
\begin{align}
\mathrm{r}_{\text{pack},\KK}(p^{\frac{r}{n}} L)&=p^{\frac{r}{n}} \, \mathrm{r}_{\text{pack},\KK}(L)\geq p^{\frac{3}{n}} 3\bar{\rho} \mathrm{r}_{\text{pack},\KK}(L)\geq 3 p^{\frac{3}{n}}  \mathrm{r}_{\text{cov},\KK}(L)\nonumber\\
&\geq 3 p^{\frac{3}{n}}  \left(\frac{1}{\Vol(\KK)}\right)^{\frac{1}{n}}=3 \left(\frac{p^3}{\Vol(\KK)}\right)^{\frac{1}{n}}>\frac{3}{2.5}\stackrel{\eqref{eq:8cnbbound}}{>}1+\frac{c}{p},
\end{align}
and upper bound the covering radius as
\begin{align}
\mathrm{r}_{\text{cov},\KK}(p^{\frac{r}{n}} L)&=p^{\frac{r}{n}} \, \mathrm{r}_{\text{cov},\KK}(L)\leq 3 p^{\frac{4}{n}}  \mathrm{r}_{\text{cov},\KK}(L)\bar{\rho}\leq  3 p^{\frac{4}{n}} \bar{\rho}^2 \mathrm{r}_{\text{pack},\KK}(L)\nonumber\\
&\leq 3 p^{\frac{4}{n}} \bar{\rho}^2 \left(\frac{1}{\Vol(\KK)}\right)^{\frac{1}{n}}=3 p^{\frac{1}{n}} \bar{\rho}^2 \left(\frac{p^3}{\Vol(\KK)}\right)^{\frac{1}{n}}\leq 3 p^{\frac{1}{n}} \bar{\rho}^2 e^{-\frac{1}{n}}<c.
\end{align}
Finally, we have that
\begin{align*}
r 
&= \log_p \covol(p^{r/n} L) \\
&= \log_p\frac{\covol(p^{r/n} L)}{\Vol(\KK)} + \log_p \Vol(\KK)\\
&\stackrel{\eqref{eq:primelimits}}{\ge} \log_p\frac{\covol(p^{r/n} L)}{\Vol(\KK)} + 3 + \log_p(e)\\
& > 3+ \log_p\frac{\covol(p^{r/n} L)}{\Vol(\KK)}+3c \, \frac{n}{p\log p},
 \end{align*}
 where we have used the fact that $\frac{3cn}{p}<1$, due to~\eqref{eq:8cnbbound}, in the last inequality. Therefore, the conditions of Theorem~\ref{thm:packtocov} apply to the lattice $p^{r/n} L$ and the convex body $\KK$, with $c=3 p^{\frac{1}{n}} \bar{\rho}^2$. Thus, for $S\sim \Unif\left(\Gr_{n,r}(\F_p)\right)$, we have that
\begin{align*}
\Pr\left(\SM(\KK,p^{r/n} L(S)) \geq \tau+8\frac{cn}{p}\right) <  2\delta.
\end{align*}
The statement in~\eqref{eq:taudeltaProbAlternative} follows since $8\frac{cn}{p}<\tau$ by~\eqref{eq:8cnbbound}. 
\end{proof}

To prove Theorem~\ref{thm:eta_siegel}, we will also need the following auxiliary statement, proved in \S\ref{sec:highprobbounds}. Recall that $\mu_n$ denotes the Haar-Siegel probability measure.

\begin{prop}\label{prop: 3.6}
For $L\sim \mu_n$, any convex body $\KK\in\Convn$, and any $\alpha>0$,
\begin{align*}
\Pr(\rho_{\KK}(L)\geq 2\cdot \alpha^2)<3\cdot\left(\frac{2}{\alpha}\right)^n.
\end{align*}
\label{prop:rhoHaarSiegel}
\end{prop}

\begin{proof}[Proof of Theorem~\ref{thm:eta_siegel} (assuming Proposition \ref{prop: 3.6})]
Let $L'\sim\mu_n$. Recall from \cite{ORW21}*{Prop.~2.1} that for  any fixed  prime $p$ and any $1\leq r\leq n$, if  we sample $S $ according to the uniform distribution on  $\Gr_{n,r}(\F_p)$, statistically independent of $L'$, 
the lattice $L=p^{r/n}L'(S)$ will also be distributed according to $\mu_n$. Thus, for $L\sim\mu_n$ and any $\tau\in(0,1)$,
\begin{align}
 \Pr(\eta(\KK,L) \geq 2\tau)=\Pr(\eta(\KK,p^{r/n}L'(S)) \geq 2\tau).  \nonumber
\end{align}
We proceed to upper bound the right hand side of the above expression, using Corollary~\ref{thm:inflatedeflateAlternative}. 
Let $\bar{\rho}=8\cdot\left(\frac{50}{\delta}\right)^{\frac{2}{n}}$ (note that $\bar{\rho}<\left(\frac{2}{\delta}\right)^{\frac{n}{2}}$ for $n>25$), and let $E$ be the set of all lattices $L'$ with unit covolume  
for which $\rho_{\KK}(L')< \bar{\rho}$. Applying Proposition~\ref{prop:rhoHaarSiegel} with $\alpha=2\left(\frac{50}{\delta}\right)^{\frac{1}{n}}$, we have that $\Pr(L'\in E^c)<0.06\delta$. Now, applying Corollary~\ref{thm:inflatedeflateAlternative} with $\delta'=0.47\delta$, $\tau=\eps/2$, we see that for any $\KK\in\Convn$ with 
\begin{align}
\Vol(\KK)&>\left(\frac{2}{0.47}\right)^6 c_5 \bar{\rho}^6 \left(\frac{1}{\eps\delta}\right)^6 n^3\nonumber\\
&=\left(\frac{2}{0.47}\right)^6\cdot 8^6\cdot 50^{\frac{12}{n}}\cdot c_5 \cdot\left(\frac{1}{\delta}\right)^{\frac{12}{n}}\left(\frac{1}{\eps\delta}\right)^6 n^3,\label{eq:volkklb}
\end{align}
there is a prime number $p$ and an integer $3<r<n$ for which
\begin{align}
\Pr\left(\eta(\KK,p^{r/n}L'(S))\geq \eps ~|~ L'\in E \right)<0.94\delta.\nonumber
\end{align}
In particular, this holds for any $\KK\in\Convn$ with $\Vol(\KK)>c_1 \left(\frac{1}{\eps\delta}\right)^{6.5} n^3$, since $c_1 \left(\frac{1}{\eps\delta}\right)^{6.5} n^3$ is greater than the right hand side of~\eqref{eq:volkklb}.
Our claim now follows since
\begin{align}
 \Pr\left(\eta(\KK,L) \geq \eps \right)&=\Pr\left(\eta(\KK,p^{r/n}L'(S)) \geq \eps \right )\nonumber\\
 &\leq \Pr (L'\in E^c)+\Pr\Big(\eta(\KK,p^{r/n}L'(S)) \geq \eps ~|~ L'\in E \Big)\nonumber\\
 &< 0.06\delta+0.94\delta.\nonumber
\end{align}
\end{proof}

\begin{remark}\label{rem:ddparams}
Improved bounds in Theorem~\ref{thm: DD} will result in tighter upper bounds on the minimal required volume for smooth covering. Specifically, assume the following holds: for $n$ large enough, $\delta,\tau\in(0,1)$ any prime $p> p^*$ and any $A\subset \F_p^n$ and $r>m_*+(n-\log_p |A|)$, the conclusion of Theorem~\ref{thm: DD} holds.
Then, roughly speaking, the proof we give for Theorem~\ref{thm:eta_siegel}, with simple modifications, shows that there is a  constant $c>0$, such that \eqref{eq:SiegelProbBound} holds for any convex body $\KK$ for which 
$\Vol(\KK)\geq c\cdot (p_*)^{m_*}$, as long as 
$p_*= \Omega(n)$. 

On the other hand, it follows from \cite{CoxeterFewRogers} that there exists $\KK\in\Convn$ with volume $\Omega(n)$ such that $\eta(\KK,L)\geq 1$ for all unit covolume lattices. This gives an obvious bound on the extent to which Theorem~\ref{thm:eta_siegel} can be improved. Namely, if one can prove that~\eqref{eq:ddbound} holds for fixed $0<\delta,\tau<1$, $p_*\asymp n$ and  $m_*$ arbitrarily close to $1$, this will show that the lower bound in~\cite{CoxeterFewRogers} is essentially tight, and is attained for a ``typical'' lattice (and even for $\eta<1$, i.e., with smooth covering). 
\end{remark}

For the proof of Theorem~\ref{thm:eta_constA} we will also need the following statement, proved in 
\S\ref{subsec:mono}.
\begin{lem}
\label{lem:monotonicity}
 For any lattice $L\subset\R^n$, convex set $\KK\in\Convn$, and positive integer $m$, we have
 \begin{align}
 \eta(m\KK,L)\leq \eta(\KK,L).\nonumber
 \end{align}
\end{lem}

\begin{proof}[Proof of Theorem~\ref{thm:eta_constA} (assuming Lemma \ref{lem:monotonicity})]
Let 
$$M=c_5\cdot 4^6 \cdot \left(\frac{1}{\eps\delta}\right)^{6}n^{3(1+2b)}=c_{3}\left(\frac{1}{\eps\delta}\right)^{6}n^{3(1+2b)},$$ so that by assumption we have $\Vol(\KK)\geq M$. 
By Lemma~\ref{lem:monotonicity}, replacing $\KK$ if necessary with $\frac{1}{a}\KK$ for some integer $a\geq 2$, we may further assume $\Vol(\KK)\in[M,2^nM)$. We apply Corollary~\ref{thm:inflatedeflateAlternative} with $L=\Z^n$, $\bar{\rho}=n^b$, $\delta'=\delta/2$ and $\tau=\epsilon/2$.  
It is straightforward to verify that $p$ satisfies~\eqref{eq:primelimits} for  $\Vol(\KK)\in[M,2^nM)$, and the conditions on $\Vol(\KK)$ and $r$ also trivially hold. Thus, recalling that $p^{r/n}L(S)=p^{r/n}\Z^n(S)$ is distributed as a lattice drawn from the $(p,r)$ random construction A ensemble, we obtain the required statement. 
\end{proof}

\subsection{High probability bounds on \texorpdfstring{$\rho_\KK(L)$}{rho\_K(L)}}
\label{sec:highprobbounds}
In this subsection we will give a simple proof of Proposition \ref{prop: 3.6}. The first results showing the existence of a global constant $c>0$, independent of the dimension $n$, such that for any $\KK \in \Convn$ there is a lattice $L$ such that $\rho_{\KK}(L)<c$, are due to Butler \cite{Butler} and Bourgain \cite{Bourgain}.  The  probability (with respect to the Siegel-Haar measure) that a randomly chosen lattice satisfies $\rho_{\KK}(L)<c$ was not discussed in these papers. Using 
\cite{ORW21}*{Cor.~1.6} together with the bound $\Vol(\KK-\KK)\leq \Vol(4\KK)$ proved in \cite{RogersShepard57}, one sees that $\Pr\left(\rho_{\KK}(L)>4+o(1) \right)$ vanishes exponentially fast with $n$. However, we can give a much simpler proof, albeit with a worse constant. Note that the value of the constant  has small effect on the bounds we obtain for the covering smoothness.

\begin{proof}[Proof of Proposition \ref{prop: 3.6}]
We will derive a high-probability lower bound on $\mathrm{r}_{\text{pack},\KK}(L)$ and a high-probability upper bound on $\mathrm{r}_{\text{cov},\KK}(L)$. Assume without loss of generality that $\Vol(\KK)=1$. 
Denote
$$
N^*(L, \KK, x) \df  \left|((L\setminus\{0\})-x) \cap \KK\right| = \left| \left\{y\in L\setminus\{0\} \ : \ x\in y-\KK\right\}\right|.
$$
For lower bounding $\mathrm{r}_{\text{pack},\KK}(L)$, let $\KK_0=\frac{1}{2\alpha}\KK$ and let $N^*(L,\KK_0-\KK_0,0)$ be the number of non-zero points of the lattice $L$ in $\KK_0-\KK_0$. By Siegel's theorem~\cite{SiegelFormula}, we have  $$\mathbb{E}[N^*(L,\KK_0-\KK_0,0)]=\Vol(\KK_0-\KK_0).$$ 
Thus, by Markov's inequality,
\begin{align}
\Pr(N^*(L,\KK_0-\KK_0,0)&\neq 0)=\Pr(N^*(L,\KK_0-\KK_0,0)\geq 1)\nonumber\\
&\leq \mathbb{E}[N^*(L,\KK_0-\KK_0,0)]\nonumber\\
&=\Vol(\KK_0-\KK_0)\nonumber\\
&\leq \Vol(4\KK_0)\label{eq:RogesrShepard} =\left(\frac{4}{2\alpha}\right)^n, 
\end{align}
where the last inequality in~\eqref{eq:RogesrShepard} follows from~\cite{RogersShepard57}*{Theorem 1}. 
Now, since $N^*(L,\KK_0-\KK_0,0)=0$ implies that $L$ forms a packing with respect to $\KK_0$, we see that
\begin{align}
\Pr\left(\mathrm{r}_{\text{pack},\KK}(L)\leq \frac{1}{2\alpha}\right)\leq \left(\frac{2}{\alpha}\right)^n.
\label{eq:packingwhp}
\end{align}
Recall that $m_L$ denotes the Haar measure on $\Tor_L=\mathbb{R}^n/L$.
For the upper bound on the covering radius 
recall the basic fact (see, e.g., \cite{ORW21}*{Lemma 2.5}) that $$  m_L\left(\pi_L\left(\frac{\alpha}{2} \KK\right ) \right)>\frac12 \  \implies  \ \mathrm{r}_{\text{cov},\KK}(L)\leq \alpha.$$ 
 It therefore only remains to upper bound $\Pr\big(m_L(\pi_L(\KK_1))\leq 1/2\big)$, where $\KK_1=\frac{\alpha}{2} \KK$. We do this by showing that $\mathbb{E}[m_L(\pi_L( \KK_1))]$ is close to $1$. We begin by noting that
\begin{align*}
m_L(\pi_L(\KK_1))&=\int_{x\in \KK_1}\frac{1}{|(x+L)\cap \KK_1|}dx\nonumber\\
&= \int_{x\in \KK_1}\frac{1}{1+N^*(L,\KK_1,-x)}dx.
\end{align*}
Since the function $t\mapsto \frac{1}{1+t}$ is convex in the regime $t>0$, we can apply Jensen's inequality and obtain
\begin{align}
\mathbb{E}[m_L(\pi_L(\KK_1))]
&= \mathbb{E}\left[\int_{x\in \KK_1}\frac{1}{1+N^*(L,\KK_1,-x)}dx\right]\nonumber\\
&= \int_{x\in \KK_1}\mathbb{E}\left[\frac{1}{1+N^*(L,\KK_1,-x)}\right]dx\label{eq:fubini}\\
&\geq \int_{x\in \KK_1}\frac{1}{1+\mathbb{E}\left[N^*(L,\KK_1,-x)\right]}dx\label{eq:jensen}\\
&=\int_{x\in \KK_1}\frac{1}{1+\Vol(\KK_1)}dx\label{eq:siegel}\\
&=\frac{\Vol(\KK_1)}{1+\Vol(\KK_1)},\label{eq:Eml_lb}
\end{align}
where~\eqref{eq:fubini} follows from Fubini's Theorem,~\eqref{eq:jensen} from Jensen's inequality and~\eqref{eq:siegel} from Siegel's summation formula. Let $$P_e= \Pr\left(m_L(\pi_L(\KK_1))\leq \frac12\right).$$ Since $m_L(\pi_L(\KK_1))\leq 1$ we have that
\begin{align*}
\mathbb{E}[m_L(\pi_L(\KK_1))]\leq \frac{P_e}{2}+(1-P_e)=1-\frac{P_e}{2}.
\end{align*}
Combining this with~\eqref{eq:Eml_lb}, we obtain $$P_e \leq 2\left(1-\frac{\Vol(\KK_1)}{1+\Vol(\KK_1)}\right)<\frac{2}{\Vol(\KK_1)}.$$
Thus,
\begin{align}
\Pr(\mathrm{r}_{\text{cov},\KK}(L)\geq\alpha)<2\cdot \left(\frac{2}{\alpha}\right)^n.
\label{eq:coveringwhp}
\end{align}
Combining~\eqref{eq:packingwhp} and~\eqref{eq:coveringwhp} we obtain the claimed result.
\end{proof}

\subsection{On the monotonicity of \texorpdfstring{$\alpha\mapsto\eta(\alpha\KK,L)$}{r maps to eta(rK,L)}}
\label{subsec:mono}

As mentioned above, the mapping $\alpha\mapsto\eta(\alpha\KK,L)$ is not monotonically non-increasing in general.
Nevertheless, Lemma~\ref{lem:monotonicity}, stated above, shows that for dilates by positive integers, the covering smoothness can only decrease. We can exploit this fact to establish Theorems~\ref{thm:Phi_siegel} and~\ref{thm:Phi_constA}, which show that for any $\KK\in\Convn$ with sufficiently large (polynomial) volume, a typical lattice has small $\eta(\alpha\KK,L)$ for all $\alpha\geq 1$. We first provide the proof of Lemma~\ref{lem:monotonicity}, and then leverage this result and prove Theorems~\ref{thm:Phi_siegel} and~\ref{thm:Phi_constA}.

\begin{proof}[Proof of Lemma~\ref{lem:monotonicity}]
We can write 
\[
N(L, m \KK, x) = 
N\left(\frac{L}{m}, \KK, \frac{x}{m}\right) = 
\sum_{a} N\left(L, \KK, \frac{x}{m}-a\right) \; ,
\]
where the sums runs over all coset representatives $a$ for the inclusion $L \subset \frac{L}{m}$.
We therefore have
\begin{align*}
\SM(m\KK,L) &= 
\sup_{x\in\R^n}\left|\frac{N(L,m\KK,x)}{\Vol(m\KK)/\covol(L)}-1\right| \\ 
&=
\sup_{x\in\R^n}\left| m^{-n} \sum_a \Big( \frac{N(L,\KK,x/m-a)}{\Vol(\KK)/\covol(L)}-1 \Big) \right| \\
&\le
m^{-n} \sup_{x\in\R^n} \sum_a  \left| \frac{N(L,\KK,x/m-a)}{\Vol(\KK)/\covol(L)}-1 \right| \\
&\le 
m^{-n} \sum_a \sup_{x\in\R^n} \left| \frac{N(L,\KK,x/m-a)}{\Vol(\KK)/\covol(L)}-1 \right| =
\SM(\KK,L) \; .
\end{align*}
\end{proof}


Using this weak monotonicity property, we now show that if a lattice $L$ smoothly covers $\R^n$ with respect to $\alpha_i\KK$ for all $\alpha_i$ in a dense enough net in $[1,2)$, it must smoothly cover $\R^n$ with respect to $\alpha\KK$ for all $\alpha\geq 1$.

\begin{lem}
Let $n\in\N$ and $0<\eps<1$. Let $\beta=\frac{\eps}{8n}$ and $I=\left\lceil \frac{\log 2}{\log(1+\beta)}\right\rceil$. Define $\alpha_i=(1+\beta)^i$ for all $i=0,1,\ldots,I$, such that $\alpha_0=1$, and $\alpha_I\geq 2$. For a lattice $L\subset\R^n$ and $\KK\in\Convn$ assume that $\SM(\alpha_i\KK,L)\leq \eps/2$ for all $i=0,1,\ldots,I$. Then  $\SM(\alpha\KK,L)< \eps$ for all $\alpha\geq 1$.
\label{lem:netcoveringimpliesfullcovering}
\end{lem}

\begin{proof}
Note that for any $\alpha\in[1,2)$ there is $i\in\{1,\ldots,I\}$ such that $\alpha_{i-1}\leq \alpha\leq \alpha_i$. We therefore have that for any $x\in\R^n$,
\begin{align*}
 \frac{N(L,\alpha\KK,x)}{\Vol(\alpha\KK)/\covol(L)}&\leq \frac{\Vol(\alpha_i\KK)}{\Vol(\alpha\KK)}\frac{N(L,\alpha_i\KK,x)}{\Vol(\alpha_i\KK)/\covol(L)}\\
 &\leq (1+\beta)^n\left(1+\frac{\eps}{2}\right)<(1+\eps),
\end{align*}
where in the last inequality we used the fact that $(1+\beta)^n\leq e^{\beta n}\leq 1+2\beta n=1+\frac{\eps}{4}$, which follows since $e^t<1+2t$ for $t<1/2$. Similarly,
\begin{align*}
 \frac{N(L,\alpha\KK,x)}{\Vol(\alpha\KK)/\covol(L)}&\geq \frac{\Vol(\alpha_{i-1}\KK)}{\Vol(\alpha\KK)}\frac{N(L,\alpha_{i-1}\KK,x)}{\Vol(\alpha_{i-1}\KK)/\covol(L)}\\
 &\geq (1+\beta)^{-n}\left(1-\frac{\eps}{2}\right)
 >(1-\eps),
\end{align*}
where, as above, in the last inequality we used the fact that $(1+\beta)^{-n}\geq \frac{1}{1+2\beta n}\geq 1-2\beta n=1-\frac{\eps}{4}$. Thus, $\SM(\alpha\KK,L)<\eps$ for all $\alpha\in[1,2)$. Finally, for any $\alpha\geq 1$ there is a positive integer $m$ such $\alpha'=\alpha/m \in[1,2)$, and thus, by Lemma~\ref{lem:monotonicity}, $\SM(\alpha\KK,L)<\eps$.
\end{proof}

\begin{proof}[Proof of Theorem~\ref{thm:Phi_siegel}]
Assume $\Vol(\KK)=c_2 \left(\frac{1}{\eps^2\delta}\right)^{6.5} n^{3}$. 
Let $$\beta=\frac{\eps}{8n} \text{ and } I=\left\lceil \frac{\log 2}{\log(1+\beta)}\right\rceil\leq\frac{8n}{\eps}-1.$$ 
Define $\alpha_i=(1+\beta)^i$ for all $i=0,1,\ldots,I$ and note that
\begin{align}
\alpha_i^n\geq  e^{\frac{\eps}{9}i}.
\label{eq:alphaibound}
\end{align}
Indeed,
\begin{align*}
(1+\beta)^n=\left(1+\frac{\eps}{8n}\right)^{n}=e^{ n \log(1+\frac{\eps}{8n})}\geq e^{ n \frac{\eps/8n}{1+\eps/8n}}=e^{\frac{\eps/8}{1+\eps/8n}}>e^{\frac{\eps/8}{9/8}}=e^{\frac{\eps}{9}}.
\end{align*}
For all $i=0,1,\ldots,I$ we apply Theorem~\ref{thm:eta_siegel} with $\eps'=\eps/2$ and $\delta_i=\frac{\delta\eps}{64}\cdot e^{-\frac{\eps}{60}i}$ and $\KK_i=\alpha_i \KK$. Noting that 
\begin{align*}
c_1\left(\frac{1}{\eps'\delta_i}\right)^{6.5} n^{3}&<(128)^{6.5} c_1  \left(\frac{1}{\eps^2\delta}\right)^{6.5} e^{\frac{\eps}{9}i} n^{3}\stackrel{\eqref{eq:alphaibound}}{<}\alpha_i^n  c_2  \left(\frac{1}{\eps^2\delta}\right)^6 n^{3}\leq \Vol(\KK_i),   
\end{align*}
the theorem implies that
\begin{align}
\Pr\left(\SM(\alpha_i\KK,L) \geq \frac{\eps}{2} \right)<\delta_i,~~~\forall i=0,1,\ldots,I.
\label{eq:riPe}
\end{align}
Let $E$ be the set of all unit covolume lattices such that $\SM(\alpha_i\KK,L)< \eps/2$ for all $i=0,1,\ldots,I$. 
By the union bound and~\eqref{eq:riPe}, we have that 
\begin{align*}
\Pr(L\notin E)&\leq \sum_{i=0}^I \delta_i=\frac{\delta\eps}{64}\sum_{i=0}^I e^{-\frac{\eps}{60}i}< \frac{\delta\eps}{64}\sum_{i=0}^\infty e^{-\frac{\eps}{60}i}=\frac{\delta\eps}{64}\frac{1}{1-e^{-\frac{\eps}{60}}}\\
&\leq \frac{\delta\eps}{64}\left(1+\frac{60}{\eps}\right)<\delta,
\end{align*}
where we have used the fact that $e^{-x}\leq \frac{1}{1+x}$ for $x\geq 0$.
Our claim now follows by applying Lemma~\ref{lem:netcoveringimpliesfullcovering}.
\end{proof}

\begin{proof}[Proof of Theorem~\ref{thm:Phi_constA}]
Assume $\Vol(\KK)=c_4 \left(\frac{1}{\eps^2\delta}\right)^6 n^{9+6b}$. 
Let $\beta=\frac{\eps}{8n}$ and $I=\left\lceil \frac{\log 2}{\log(1+\beta)}\right\rceil\leq \frac{8n}{\eps}-1$. Define $\alpha_i=(1+\beta)^i$ for all $i=0,1,\ldots,I$. For all $i=0,1,\ldots,I$ we apply Theorem~\ref{thm:eta_constA} with $\eps'=\eps/2$ and $\delta'=\delta/(I+1)$ and $\KK'=\alpha_i \KK$. Noting that $\eps'\delta'\geq \frac{\eps^2\delta}{16n}$ and that
\begin{align*}
c_3\left(\frac{1}{\eps'\delta'}\right)^6 n^{3(1+2b)}\leq  c_4 \left(\frac{1}{\eps^2\delta}\right)^6 n^{9+6b}\leq \Vol(\KK'), 
\end{align*}
the theorem implies that
\begin{align}
\Pr(\SM(\alpha_i\KK,p^{r/n}L) \geq \eps/2)<\frac{\delta}{I+1},~~~\forall i=0,1,\ldots,I.
\label{eq:riPeConstA}
\end{align}
Let $E$ be the set of all $(p,r)$ construction A lattices such that $$\SM(\alpha_i\KK,p^{r/n}L)<\frac{\eps}{2}, \ \  \text{  for } i=0,1,\ldots,I.$$ 
By the union bound and~\eqref{eq:riPeConstA}, we have that $\Pr(L\notin E)<\delta$. Our claim now follows by applying Lemma~\ref{lem:netcoveringimpliesfullcovering}.
\end{proof}

\subsection{Non-lattice smooth coverings}\label{sec:nonlattice}
In this subsection we will prove Theorem \ref{thm: non lattice}. The proof follows the same outline and notation  as the proof of Theorem \ref{thm:packtocov}. In the previous sections we started with a lattice $L$ with a reasonable $\rho_{\KK}(L)$ and constructed from it a denser lattice $L(S)$ by choosing $S\subset \F_p^n$ to be a subspace. 
The work of Dhar and Dvir~\cite{dd22} was then used to show that for any subset $A\in\F_p^n$, and a randomly uniform subspace $S\in\F_p^n$, if $p$ is sufficiently large and $\frac{|S|\cdot |A|}{p^n}>p^3$, then $\eta_{\F_p}(A,S)$ is small with high probability. This was then leveraged for showing that under suitable conditions a randomly chosen subspace $S$ will yield a lattice $L(S)$ such that $L(S)+\KK$ smoothly covers $\R^n$. Note that if $S\subset \F_p^n$ is not a subspace, the discrete set $L(S)$, as given in~\eqref{eq:LSdef} is not a lattice, but is nevertheless well defined, and has asymptotic density $\frac{|S|}{\covol(L)}$. Furthermore, recall that the definition of $\eta_{\F_p}(A,S)$ does not require $S$ to be a subspace. For a random set $S$ (rather than a random subspace, as in~\cite{dd22}), controlling the tail of $\eta_{\F_p}(A,S)$ is a significantly simpler task. The following result easily follows from large deviation theory.

\begin{lem}
Let $n,m$ be positive integers, $p$ be a prime number and let $\delta,\tau\in(0,1)$. Let $S$ be a set of $m$ points identically distributed independently uniformly over $\F_p^n$. Then, for any set $A\subset\F_p^n$, we have that if $\frac{m|A|}{p^n}>\frac{3}{\tau^2}\left(n\log p-\log\frac{\delta}{2}\right)$, then
    \begin{align*}
	\Pr(\SMfp(A,S) \geq \tau~\text{or}~ |S|\neq m) \leq \delta+m^2p^{-n}.
	\end{align*}
	\label{lem:etaFrand}
\end{lem}

The lemma follows easily from  the following well-known large deviations bound, see e.g.,~\cite{MitzenmacherUpfal17}*{Corollary 4.6}.
\begin{prop}[Chernoff bound]\label{prop: Chernoff}
For any $\eta \in (0,1)$ and any $m$ identically distributed independent Bernoulli random variables $Y_1, \dots, Y_m$, the sum $Y \df \sum_{i=1}^m Y_i $ satisfies 
\begin{equation}\label{eq: this ensures}
\Pr\left( \left| Y  - \mu \right| \geq \eta \mu    \right) \le 2e^{-\frac{\eta^2\mu}{3}},
\end{equation}
where $\mu \df \mathbb{E}(Y).$
\end{prop}
\begin{proof}[Proof of Lemma \ref{lem:etaFrand}]
Let $X_i\stackrel{i.i.d}{\sim} \Unif(\F_p^n)$ for $i=1,\ldots,m$, and $S=\{X_1,\ldots,X_m\}$. For any $x\in\F_p^n$ let $Y_{i,x}$ be the indicator of the event that $x+X_i\in A$. We clearly have that the random variables $\{Y_{i,x}\}_{i=1}^m$ are i.i.d.\ Bernoulli with $\Pr(Y_{i,x}=1)=\frac{|A|}{p^n}$. Thus, $Y_x=\sum_{i=1}^m Y_{i,x}$  satisfies the conditions of Proposition~\ref{prop: Chernoff}, and applying it with $\mu=\mathbb{E}(Y_x)=\frac{m|A|}{p^n}$ and $\eta=\tau$, gives that if $\frac{m|A|}{p^n}>\frac{3}{\tau^2}\left(n\log p-\log\frac{\delta}{2}\right)$ then
\begin{align}
 \Pr&\left(\left|\frac{|(x+S)\cap A|}{m\cdot|A|p^{-n}}-1 \right|\geq \tau \right)\nonumber\\
 &=\Pr(|Y_x-\mu|\geq \tau\mu)\leq 2 e^{-\frac{\tau^2}{3}\frac{3}{\tau^2}\left(n\log p-\log\frac{\delta}{2}\right)}=\delta\cdot p^{-n}.\nonumber
\end{align}
Applying the union bound, this implies that
\begin{align}
 \Pr\left(\max_{x\in\F_p^n}\left|\frac{|(x+S)\cap A|}{m\cdot|A|p^{-n}}-1 \right| \geq \tau \right)\leq \delta.   \nonumber
\end{align}
Finally, noting that 
\begin{align*}
\Pr(|S|\neq m)
\le
\sum_{1 \le i < j \le m} \Pr(X_i=X_j)
=
\binom{m}{2} p^{-n}
< 
m^2 p^{-n},
\end{align*}
and applying the union bound again, we obtain the claimed result.
\end{proof}

\begin{proof}[Proof of Theorem \ref{thm: non lattice}]
Let $\KK\in\Convn$, $\frac{320n}{\tau}<p<\frac{640n}{\tau} $ be a prime number for some $\tau\in(0,1)$ to be chosen later, and $c=40$. Let $L$ be a lattice so that $(L,(1+\frac{c}{p})\KK)$ is a packing and $(L, c\KK)$ is a covering. Such a lattice exists by Proposition~\ref{prop: 3.6}. Note further, that for such a lattice we have that $\frac{\covol(L)}{\Vol(\KK)}<c^n=40^n$. Denote $\rho=\frac{c}{p}<\frac{\tau}{8n}$, such that in particular $0<\rho<\frac{1}{2n}$. We follow the derivations in the proof of Theorem~\ref{thm:packtocov} up to equation~\eqref{eq:etaEc}, where instead of assuming $S\in \Gr_{n,r}(\F_p)$, we assume $S\subset\F_p^n$ is an arbitrary subset of $m$ points in $\F_p^n$. This derivation does not rely on $S$ being a subspace and therefore holds verbatim, where the only difference is that we replace the definitions of the sets $E_0$ and $E_1$ from~\eqref{eq:E0def} and \eqref{eq:E1def} with
\begin{align*}
E_0&\df\left\{S\in \mathcal{R}_{m,n}(\F_p) \ : \   \SMfp(A_0,S)>\tau  \right\},\\
E_1&\df\left\{S\in \mathcal{R}_{m,n}(\F_p) \ : \   \SMfp(A_1,S)>\tau  \right\},
\end{align*}
where $\mathcal{R}_{m,n}(\F_p)=\{S\subset \F_p^n~:~|S|=m\}$. We therefore have that 
\begin{align}
\SM(\KK,L(S))\leq \tau+8\rho n\leq 2\tau,~~ \forall S\in E^c.
\label{eq:etaEcInfinite}
\end{align}
We proceed to upper bound $\Pr(S\in E)$ for the case where $S$ consists of $m$ points drawn i.i.d.\ from the uniform distribution over $\F_p^n$.
By~\eqref{eq:sizeAbounds}, we have that for $i=0,1$
\begin{align}
\frac{|A_i|}{p^n}\geq (1-2\rho)^n\frac{\Vol(\KK)}{\covol(L)}\geq (1-2n\rho)\frac{\Vol(\KK)}{\covol(L)}\geq\left(1-\frac{\tau}{4}\right)\frac{\Vol(\KK)}{\covol(L)},
\end{align}
where the second inequality is due to~\eqref{eq:1mrhoLB}, and the third follows since $\rho<\frac{\tau}{8n}$. Thus, by Lemma~\ref{lem:etaFrand}, for any $\delta\in(0,1)$, if
\begin{align}
\left(1-\frac{\tau}{4}\right)\cdot \Vol(\KK)\cdot\frac{m}{\covol(L)}>\frac{3}{\tau^2}\left(n\log p-\log\frac{\delta}{2}\right)
\end{align}
then $\Pr(E_i)<\delta+m^2 p^{-n}$ for $i=1,2$. We take $\tau=\eps/2$ and $\delta=2e^{-2}$ and choose 
\begin{align*}
m&=\left\lceil \frac{\covol(L)}{\Vol(\KK)}\frac{1}{\left(1-\frac{\eps}{8}\right)}    \frac{12}{\eps^2}\left(n\log p+2\right) \right\rceil \\
&< \frac{\covol(L)}{\Vol(\KK)} \frac{14}{\eps^2}\left(n\log p+2\right)
\end{align*}
to be the smallest integer satisfying the above constraint, so that $\Pr(E)\leq 2\delta+2{m}^2p^{-n}=4 e^{-2}+2{m}^2p^{-n}$. Recalling that $\frac{\covol(L)}{\Vol(\KK)}< 40^n$ and that $\frac{640n}{\eps}\leq p\leq\frac{1280 n}{\eps}$ we see that
\begin{align*}
m^2 p^{-n}&\leq 40^{2n} \frac{14^2}{\eps^4}\left(n\log \frac{2560n}{\eps}\right)^2 \left(\frac{640 n}{\eps}\right)^{-n}\\
&<40^{2n} \frac{14^2}{\eps^4}\left(\frac{2560n^2}{\eps}\right)^2 \left(\frac{640 n}{\eps}\right)^{-n}\\
&<\left(\frac{40^2}{640}\right)^n \left(\frac{n}{\eps}\right)^{-(n-6)}(14\cdot 2560)^2
\end{align*}
and this is smaller than $0.08$ for all $n\geq 20$, and so $\Pr(E)\leq 4 e^{-2}+2m^2p^{-n}<1$ for all $n\geq 20$. We therefore see that there exists a discrete set $L(S)$ with asymptotic density $D(L(S))=\frac{m}{\covol(L)}$ such that 
\begin{align}
\Vol(\KK) D(L(S))<\frac{14}{\eps^2}\left(n\log p+2\right)
\end{align}
and $\eta(\KK,L(S))<\eps$. Recalling that $p<\frac{1280n}{\eps}$, we obtain the claimed result.
\end{proof}

\bibliographystyle{alpha}
\bibliography{elonbib}

\end{document}